\documentclass{article}

\usepackage{amsfonts,amssymb,amsmath,amsthm}
\usepackage{comment}

\newtheorem{theorem}{Theorem}

\newtheorem{proposition}{Proposition}

\newcommand{\N}{{\mathbb N}}

\DeclareMathOperator{\ho}{H} 
\DeclareMathOperator{\sep}{sep}
\DeclareMathOperator{\abssep}{abs\,sep}

\usepackage{xcolor}

\begin{document}

\title{{\Large{\bf Absolute root separation}}}

\author{Yann Bugeaud, Andrej Dujella, Wenjie Fang,\\ Tomislav
Pejkovi\'{c}, and Bruno Salvy}

\date{\today}
\maketitle

\begin{abstract}
The absolute separation of a polynomial is the minimum nonzero
difference between the absolute values of its roots. In the case of
polynomials
with integer coefficients, it can be bounded from below in terms of the
degree and the height (the maximum absolute value of the coefficients)
of the polynomial. We improve the known bounds for this problem and
related ones. Then we report on extensive experiments in low degrees,
suggesting that the current bounds are still very pessimistic.
\end{abstract}

\section{Separation and absolute separation}
The \emph{absolute separation} of a polynomial $P\in\mathbb{C}[X]$
is
the minimal nonzero distance between the absolute values of its
complex roots:
\[ \abssep(P):=\min_{\substack{P(\alpha)=P(\beta)=0,\\ |\alpha|\neq|\beta|}}\big||\alpha|-|\beta|\big|. \]
Having good lower bounds on this quantity for polynomials with
integer coefficients is of interest in the
asymptotic analysis of linear recurrent sequences.

To the best of our knowledge, the first published bound on this
problem~\cite{GourdonSalvy1996} was $\abssep(P)\gg \ho(P)^{-d
(d^2+2d-1)/2}$ where, here and below, the constant implicit in the $\gg$ sign depends
only on the degree~$d$, while $\ho(P)$, the \emph{height} of the
polynomial~$P$, is the maximum of the absolute values of its
coefficients. This exponent was later~\cite{DubickasSha2015,Sha2019}
improved
to~$-d^3/2+d^2+d/2-2$ and even more recently~\cite{Dubickas2019} to
$-d^3/2+d^2+d/2-1$.
In this work, we improve this exponent and that of related problems.
Still, we do not know how far the exponent we obtain is from
being
optimal.
Thus an important part of this article is devoted to experiments in
low degree, from where we can infer families of polynomials
exhibiting a behaviour in $\ho(P)^{-d-1}$ for $d\in\{3,4,5,6\}$.

In the much more classical case of the separation
\[ \sep(P):=\min_{\substack{P(\alpha)=P(\beta)=0,\\
\alpha\neq\beta}}|\alpha-\beta|, \]
the best available bound for a polynomial with integer coefficients
goes back to Mahler~\cite{Mahler1964}:
\begin{equation}\label{eq:Mahler}
\sep(P)\gg \ho(P)^{-d+1}.
\end{equation}
Even in that case, the tightness of the exponent $-d+1$ is still
unknown, with best known upper bounds $-(2d-1)/3$ for general $d$ and
$-2$ for $d=3$~\cite{BugeaudDujella2014,Evertse2004,Schonhage2006}.
(A consequence noted by Mahler is that the right-hand side
of~\eqref{eq:Mahler} also
gives a lower bound on
the absolute value of the imaginary parts of nonreal roots of~$P$.)

This work consists of two parts. In the first one, we improve the
known exponents for the absolute separation  and related problems with
the following.
\begin{theorem}\label{thm:main}Let $P\in\mathbb{Z}[X]$ be a polynomial
of degree~$d$
and let $\alpha$ and $\beta$ be two of its roots such that
$|\alpha|\neq|\beta|$, then
\begin{enumerate}
	\item if $\alpha$ and $\beta$ are real, then
	$\bigl||\alpha|-|\beta|\bigr|\gg \ho(P)^{-(d-1)}$;
	\item if $\alpha$ is real and $\beta$ is not, then
	$\bigl||\alpha|-|\beta|\bigr|\gg \ho(P)^{-2(d-1)(d-2)}$;
	\item if neither of them is real, then
	$\bigl||\alpha|-|\beta|\bigr|\gg \ho(P)^{-(d-1)(d-2)(d-3)/2}.$
\end{enumerate}
\end{theorem}
We proved a more precise version of the first bound in a previous
work~\cite{BugeaudDujellaPejkovicSalvy2017}, where we showed that the
exponent of $\ho(P)$ is optimal in that case. More detailed but less
precise bounds for the second case can be found in~%
\cite[Lemma~2.5]{DubickasSha2015}, \cite[Lemma~3.6]{MelczerSalvy2016}
and~\cite[Lemma~53]{MelczerSalvy2019}.
Note that the third case requires $d\ge4$ to be meaningful, since a
cubic
polynomial with real coefficients cannot have two non-real roots
with distinct absolute values.

The proof of Theorem~\ref{thm:main} is based on
constructing
auxiliary polynomials with
integer coefficients of controlled height whose roots contain the
desired difference. This is a very versatile approach. We illustrate
it to rederive Mahler's exponent $-d+1$ in Section~\ref{sec:Mahler}.
Next, we apply it to the three cases covered in Theorem~\ref{thm:main}.
Similar bounds for the difference between the real or the imaginary
parts of roots of integer polynomials are derived in
Sections~\ref{sec:realpart} and~\ref{sec:impart}.

In the second part of this work (Section~\ref{sec:experiments}), we describe
experiments leading to
lower bounds on the absolute separation for small degrees, which can
be summarized as follows.
\begin{theorem}
For each $d\in\{3,4,5,6\}$, there exists a sequence~$(P_
{d,M})$ of polynomials of degree~$d$ in~$\mathbb{Z}[X]$, such that
as $M\rightarrow\infty$, the polynomial $P_{d,M}$ has two roots $\alpha_M,\beta_M$ with
$|\alpha_M|\neq|\beta_M|$ and
\[\bigl||\alpha_M|-|\beta_M|\bigr|\ll\ho(P_{d,M})^{-d-1},\quad
M\rightarrow\infty.\]
\end{theorem}

Thus apart from its first part, $d=3$ is the 
only case where we know Theorem~\ref{thm:main} to be optimal.
It is interesting to note that, in our examples, the growth of $H(P_{3,M})$ is exponential in $M$, 
while that of $H(P_{d,M})$, for $d\in\{4,5,6\}$, is linear in $M$.

\section{General bounds from symmetric functions of roots}
Here and in the rest of this article, we consider a polynomial $P
(X)\in\mathbb{Z}[X]$,
\[P(X) = \sum_{i=0}^d a_i X^i=a_d\prod_{i=1}^d(X-\alpha_i), \] of degree
 $d$ ($a_d\neq0$) 
with coefficients of absolute value bounded
by~$H$ and complex roots $\alpha_1,\dots,\alpha_d$. 

Our bounds on
various types of separations that are asymptotic in the height of the
polynomial are obtained from the following two classical results, of
which we sketch the proofs for completeness.
\begin{proposition}[Effective Version of the Fundamental Theorem
of Symmetric Functions]\cite[Thm.~6.21]{Yap2000}\label{prop1} Let $P
(X)=\sum_{i=0}^
d{a_iX^i}=a_d\prod_{i=1}^d
{(X-\alpha_i)}\in\mathbb{Z}[a_0,\dots,a_d,X]$ and let $G\in\mathbb{Z}
[X_1,\dots,X_d]$
be a
symmetric polynomial of degree at most~$k$ in each $X_i$. Then
$a_d^kG
(\alpha_1,\dots,\alpha_d)$ is a polynomial of total degree at most~$k$
in~$
\mathbb{Z}[a_0,\dots,a_d]$.
\end{proposition}This is a standard result on symmetric
functions~\cite[Ch.~1]{Macdonald1995}. We now give an elementary
self-contained proof.

\begin{proof}
Let $b_i = a_{d-i} / a_d$ for $1 \leq i \leq d$. These are
algebraically independent symmetric functions of the $\alpha_j$'s. We
need to prove that
$G(\alpha_1, \ldots, \alpha_d)$ is a polynomial of total degree at
most $k$ in $\mathbb{Z}[b_1, \ldots, b_d]$. By linearity, it is
sufficient to consider the case where $G$ is homogeneous
of total degree $k$. Let $\Lambda_k$ be the linear space of such symmetric polynomials in $\alpha_1, \ldots, \alpha_d$.

Since $b_i$ is of total degree $i$ in the $\alpha_j$'s, all products of the form
\[ e_\mu = b_1^{\mu_1} b_2^{\mu_2} \dotsm b_{d}^{\mu_{d}} \]
indexed by $\mu = (\mu_i)$ satisfying $\sum_{i=1}^{d} i \mu_i = k$ are
linearly independent and in~$\Lambda_k$. Such a product has total
degree $\sum_{i=1}^{d} \mu_i \leq k$ in the $b_i$'s. However, $\Lambda_k$ is also linearly generated by the monomials of the form
\[ m_\lambda = \sum_{\sigma \in S_d} \alpha_{\sigma(1)}^{\lambda_1} \cdots \alpha_{\sigma(d)}^{\lambda_d} \]
indexed by $\lambda = (\lambda_j)$ with $\lambda_1 \geq \lambda_2 \geq
\cdots \geq \lambda_d$ and $\sum_{j=1}^{d} \lambda_j = k$. The
dimension is given by the number of distinct such $(\lambda_j)$, which
is also the number of $(\mu_i)$, where each $\mu_i = \lambda_i - \lambda_{i+1}$. Therefore, the $e_\mu$'s also
linearly span $\Lambda_k$ and $G$ can be written as a linear
combination of $e_\mu$'s with rational coefficients.

To prove that the coefficients are integers, we order the monomials in
$\alpha_j$'s lexicographically. 
Since $b_i$ is the sum of $\prod_{j \in S} \alpha_j$ over all subsets
$S$ of $\{1,\dots,d\}$ with $i$ elements,
the largest monomial in $e_\mu$ is given by
$\alpha_1^{\mu'_1} \cdots \alpha_d^{\mu'_d}$, where $\mu'_j = \sum_
{i=j}^{d} \mu_i$. The map $\mu\mapsto\mu' = 
(\mu'_j)$ is a bijection, which implies that every $e_\mu$ has
a distinct monic leading
monomial and all other monomials have integral coefficients.
Induction in the lexicographical order finally shows that the
coefficients of $G$ in the basis $(e_\mu)_\mu$ are all integers.
\end{proof} 

\begin{proposition}[Cauchy Bound]\cite[Thm.~4.2~(ii)]{Mignotte1992a} 
\label{lem:cauchy}\label{prop2}
  Let $P(X) = \sum_{i=0}^d a_i X^i\in\mathbb{Z}[X]$ with $|a_i|\le
  H$ for all~$i$  and let
  $\alpha\neq0\in\mathbb{C}$ be one of its roots. Then
  \[
    |\alpha|\ge\frac{1}{1+H}.
  \]
\end{proposition}
\begin{proof}
We may assume $|\alpha| < 1$, otherwise the result is obvious, and $a_0 \not= 0$. Then, we get
$$
1 \le |a_0| \le \sum_{i=1}^d |a_i| \cdot |\alpha^i| \le H \, \frac{|\alpha|}{1 - |\alpha|},
$$
giving at once the lower bound $|\alpha|\ge\frac{1}{1+H}$.
\end{proof}

\subsection{Motivating example: Mahler's bound}\label{sec:Mahler}
Mahler's bound is usually deduced from Hadamard's bound on
Sylvester's matrix applied to the discriminant of the polynomial. When
only the asymptotic exponent of~$\ho(P)$ in the estimate is needed,
Propositions~\ref{prop1} and~\ref{prop2} are sufficient. Indeed,
consider the polynomial
\[M(X)=a_d^{2(d-1)}\prod_{i<j}(X-(\alpha_i-\alpha_j)^2).\]
It is symmetric in the $\alpha_i$'s, with degree in each $\alpha_i$
that is
twice the number of $j\neq i$, i.e., $2(d-1)$. Thus, by Proposition~%
\ref{prop1}, the polynomial $M(X)$ has integer coefficients of
height bounded by~$cH^{2(d-1)}$ for some constant~$c$ that depends
only on~$d$. By Proposition~\ref{prop2}, we thus
get that for any~$(i,j)$ such that $\alpha_i\neq\alpha_j$, 
\[|\alpha_i-\alpha_j|^2\ge\frac{1}{1+cH^{2
(d-1)}}.\]
Choosing $(i,j)$ that minimizes $|\alpha_i-\alpha_j|$ and taking
square roots thus gives
\[\sep(P)\gg H^{-(d-1)},\]
recovering the exponent in Mahler's bound.

This method could be refined to produce a bound rather than an
asymptotic estimate, by a more precise analysis of the coefficients of
the polynomial involved in Proposition~\ref{prop1}. 

Note also that the polynomial obtained by considering the roots
$\alpha_i+\alpha_j$ in place of $\alpha_i-\alpha_j$ in $M$ satisfies
the same bounds. It follows that the sum of non-opposite
roots of $P$ is also
lower bounded by $H^{-(d-1)}$, giving part~1 of
Theorem~\ref{thm:main}.

\subsection{Absolute real-complex gap}
Part 2 of the theorem is obtained by considering the polynomial
\[  R(X) = a_d^{2(d-1)(d-2)} \prod_{i<j, k \notin \{i, j\}} \left(X
-
(\alpha_k^2 - \alpha_i \alpha_j)\right).\]
This polynomial is symmetric in the $\alpha_i$'s, of degree $2
(d-1)(d-2)$ in each of them.
Thus by Propositions~\ref{prop1}
and~\ref{prop2}, its smallest nonzero root has magnitude at least $H^
{-2(d-1)(d-2)}$. In particular, if $\alpha$ is a real root of $P$ and
$\beta$ a nonreal root such that $|\alpha|\neq|\beta|$, taking
$\alpha_k=\alpha$, $\alpha_i=\beta$ and $\alpha_j=\overline\beta$
gives
\[|\alpha^2-|\beta|^2|\gg H^{-2(d-1)(d-2)}.\]
If $|\alpha|+|\beta|\le2$, then dividing both terms gives a
similar inequality for $\bigl||\alpha|-|\beta|\bigr|$. 

Otherwise, it is only in the case when
$||\alpha|-|\beta||\le 1$ that we need a lower bound. 
Supposing first that $|\beta| \geq |\alpha|$, then we have
$|\alpha|+1\geq|\beta|$ and thus $|\alpha| +
|\alpha| + 1 \ge
|\alpha| + |\beta| > 2$ and so $|\alpha| > 1/2$ while $2|\beta| \geq
|\alpha| + |\beta| > 2$ gives $|\beta| > 1$. Thus, $1/|\alpha\beta| <
2$ and $1/|\alpha| + 1/|\beta|  < 3$. The same bounds are obtained
when supposing 
$|\alpha| \geq |\beta|$. 
Now,
since $\alpha^{-1}$ and
$\beta^{-1}$ are roots of the
reciprocal
polynomial~$X^dP(1/X)$ that has the same height as $P$, we also have
\[|\alpha^{-2}-|\beta|^{-2}|=\frac{\bigl||\alpha|-|\beta|\bigr|}
{|\alpha||\beta|}\left(\frac1{|\alpha|}+\frac1{|\beta|}\right)\gg H^
{-2(d-1)(d-2)},\]
whence the conclusion by dividing out by the factor smaller than~6.

\subsection{Absolute complex-complex gap}
We now analyse the polynomial
\[
  S(X) = a_d^{(d-1)(d-2)(d-3)} \prod_{\substack{i<j,\\ k<\ell,\\
  \{i,j\}\cap
  \{k,\ell\}=\emptyset}} \left(X^{1/2} - (\alpha_i\alpha_j -
  \alpha_k\alpha_\ell)\right).
\]
Exchanging the pairs $(i,j)$ and $(k,\ell)$ shows that this is indeed
a polynomial in~$X$. It is invariant under permutations of
the $\alpha_i$'s. Its degree in $X$ is $\sim d^4/4$, but its degree in
each of the $\alpha_i$'s is only $(d-1)(d-2)(d-3)$, corresponding to
all the possible choices of the other $\alpha_j$'s. 
If $\alpha$ and $\beta$ are two non-real roots of~$P$ with
$|\alpha|\neq|\beta|$, then taking $\alpha_i=\alpha$,
$\alpha_j=\overline\alpha$, $\alpha_k=\beta$,
$\alpha_\ell=\overline\beta$ and using Propositions~\ref{prop1}
and~\ref{prop2} again gives
\[(|\alpha|^2-|\beta|^2)^2\gg H^{-(d-1)(d-2)(d-3)}.\]
Taking square roots divides the exponent by~2, and then with the help of the reciprocal polynomial $X^d P(X^{-1})$, the
same argument as in the case of absolute real-complex gap leads to
\[
\abssep(P)\gg H^{-(d-1)(d-2)(d-3)/2},
\]
which concludes the proof of Theorem~\ref{thm:main}.

\subsection{Gap on the real parts of roots}\label{sec:realpart}
The same approach gives bounds on gaps between real parts of roots.
There are again three cases: real-real, real-complex, complex-complex.
The real-real case is simply the corresponding case in root
separation, which is already known. 
For the real-complex case, we consider the following polynomial
\[
  T_1(X) = a_d^{\frac{3}{2}(d-1)(d-2)} \prod_{i<j, k \notin \{i,j\}} 
  \bigl(X - (\alpha_i + \alpha_j - 2\alpha_k)\bigr),
\]
while the complex-complex case relies on
\[
  T_2(X) = a_d^{(d-1)(d-2)(d-3)} \prod_{\substack{i<j,\\ k<\ell,\\
  \{i,j\}\cap
  \{k,\ell\}=\emptyset}} \left(X^{1/2} - (\alpha_i + \alpha_j -
  \alpha_k -
  \alpha_\ell)\right).
\]
The analysis is as before. Both~$T_1$ and~$T_2$ can be seen to be
polynomials in~$X$. They are symmetric polynomials in the
$\alpha_i$'s. Their degrees in the $\alpha_i$'s is the one used in
the
exponent of~$a_d$. 
The smallest nonzero root of $T_1$ divided by 2 is a lower bound of
the real-complex real-part gap, while $T_2$ gives the complex-complex
case. Propositions~\ref{prop1} and~\ref{prop2} then give the
following.
\begin{theorem}Let $P\in\mathbb{Z}[X]$ be a polynomial of
degree~$d$ and let
$\alpha$ and $\beta$ be two of its roots such that
$\Re\alpha\neq\Re\beta$, then
\[|\Re\alpha-\Re\beta|\gg\begin{cases}
\ho(P)^{-(d-1)},\quad&\text{if $\Im\alpha=\Im\beta=0$,}\\
\ho(P)^{-3(d-1)(d-2)/2},\quad&\text{if $\Im\alpha=0$,}\\
\ho(P)^{-(d-1)(d-2)(d-3)/2},\quad&\text{otherwise.}
\end{cases}\]
\end{theorem}
\subsection{Gap on the imaginary parts of roots}\label{sec:impart}
The situation for imaginary parts is similar. First, if one of the
roots is real and the other is not, then as already mentioned,
Mahler's bound applies to the imaginary part.
If both roots are
nonreal but not conjugates, then the same polynomial $T_2$ as in the
case of the real parts can be used, with $(\alpha_i,\alpha_k)$ and $
(\alpha_j,\alpha_\ell)$ taking the roles of conjugate roots. In the
case when one of the roots is purely imaginary, then the following
generalization of~$T_1$ can be used
\[
  T_3(X) = a_d^{{3}(d-1)(d-2)} \prod_{i<j, k \notin \{i,j\}} 
  \bigl(X - (\alpha_i - \alpha_j - 2\alpha_k)^2\bigr).
\]
This discussion
leads to the following.
\begin{theorem}Let $P\in\mathbb{Z}[X]$ be a polynomial of
degree~$d$ and let
$\alpha$ and $\beta$ be two of its roots such that
$\Im\alpha\neq\Im\beta$, then
\[|\Im\alpha-\Im\beta|\gg\begin{cases}
\ho(P)^{-(d-1)},\quad&\text{if $\Im\alpha=0$,}\\
\ho(P)^{-3(d-1)(d-2)/2},\quad&\text{if $\Re\alpha=0$,}\\
\ho(P)^{-(d-1)(d-2)(d-3)/2},\quad&\text{otherwise.}
\end{cases}\]
\end{theorem}

\section{Experiments and bounds in low degree}\label{sec:experiments}
As already mentioned, even in the case of Mahler's bound, the
tightness of the exponent is unknown. The situation is similar for the
bounds obtained in the previous section. 

We now
turn to
experiments that lead to lower bounds on the asymptotic separation.
In order to obtain such an asymptotic result, we search for families
of polynomials exhibiting a small absolute separation. 
In particular, we would like to approach a tight estimate in the case
of low degrees, as a first step for a better understanding of the
actual growth of these bounds with the degree. 
We use two complementary types of experiments. 

\paragraph{Exhaustive search} First, we perform an
exhaustive search for polynomials of small absolute separation given a
degree and a bound on the height. More
specifically, given a degree, we search through all polynomials
with integer coefficients within
a given height, in case a pattern can be discerned in the
polynomials with small absolute separation. In some cases indeed, the
roots of these polynomials seem to concentrate in certain locations,
letting us refine the experiment and search more closely. This is
successful in degree~3, where we find a family of polynomials letting
us prove the tightness of the exponent~$-4$ in that case 
(Proposition~\ref{prop:deg3} below). However, even if some time is
saved by
taking into account various symmetries, the number of
polynomials to be tested is too large for this approach to be
used for large height and even more so for large
degree. 

\paragraph{Perturbations}
While these exhaustive searches are purely numerical, our
second type of experiments relies heavily on symbolic computation. 
In degree~$d$, we consider polynomials of the form
\[
	P_d(X,\epsilon) = R_d(X) + \epsilon Q_d(X).
\]
Here, for a given $r \in \N^+$, $R_d(X)$ is a polynomial with
several roots of absolute value~$r$, having factors of the type
$X \pm r$ and $X^2 + aX + r^2$, with $a$ an integer such that
$|a|<2r$. The polynomial $Q_d(X)$ is subject to $\max(\deg
(R_d),\deg(Q_d)) = d$ so that $P_d(X,\epsilon)$ has degree $d$ in~$X$.
For
$X_1
(\epsilon)$ and $X_2(\epsilon)$ two roots of $P_d(X,\epsilon)$ such that
$X_1(0)$ and $X_2(0)$ are roots of selected factors of~$R_d$, we
compute a series expansion of these roots of~$P_d$ in powers
of~$\epsilon$, from which we deduce
\[
    |X_i(\epsilon)|^2 = \sum_{k \geq 0} c_{i,k} \epsilon^k.
\]
This computation is purely symbolic, meaning that the $c_{i,k}$'s
are computed as explicit polynomials in the coefficients of $R_d$
and $Q_d$. We then look for integer solutions of the polynomial
system $\{ c_{1,k} = c_{2,k}, 1 \leq k \leq h-1\}$ that do not satisfy
$c_{1,h}=c_{2,h}$. Such
solutions give a polynomial $P_d(X,\epsilon)$ such that $|X_2
(\epsilon)|^2 -
|X_1(\epsilon)|^2$ is not zero but is $O(\epsilon^{h})$, so
that the polynomial $P_d^*(X) = M P_d(X, M^{-1})$ with integer $M$ has
integer coefficients and  
absolute separation $O(\ho(P_d^*)^{-h})$.

\paragraph{Results} The rest of this section reports on these
experiments. The
results of an exhaustive search are first presented in~\S%
\ref{sec:exhaust}. Next, we discuss a family of polynomials of
degree~3 proving the tightness of the exponent~$-4$ by an argument we
have been unable to generalize. In degrees~4, 5 and~6, the
perturbative approach lets us find exponents~$-5,-6,-7$. 

\subsection{Exhaustive search for small degree and height}
\label{sec:exhaust}

\begin{table}
\begin{tabular}{ccrl}
\hline\\[-2.2ex]
degree,&separation&$\hfill P(X)\hfill$&separation\\[-1ex]
max height&type&&\\
\hline
3,10&$|\alpha-\beta|$&$5X^3+8X^2-9X+2$&1.421e-2\\
3,20&$|\alpha-\beta|$&$14X^3+17X^2-13X+2$&4.938e-3\\
4,10&$|\alpha-\beta|$&$3X^4-9X^3-10X^2+7X-1$&4.187e-3\\
4,20&$|\alpha-\beta|$&$9X^4-13X^3-14X^2+17X-4$&5.974e-4\\
5,10&$|\alpha-\beta|$&$9X^5+X^4-4X^3-9X^2-3X+7$&4.656e-4\\
\hline 
3,10&$||\alpha|-|\beta||$&$10X^3-3X^2-2X+3$&5.394e-4\\
3,20&$||\alpha|-|\beta||$&$17X^3-9X^2-7X+8$&1.233e-5\\
4,10&$||\alpha|-|\beta||$&$X^4-6X^3-7X^2+5X+6$&2.276e-6\\
4,20&$||\alpha|-|\beta||$&$5X^4-17X^3-20X^2+11X+12$&1.034e-7\\
5,10&$||\alpha|-|\beta||$&$9X^5-5X^4-4X^3-2X^2-2X-9$&1.459e-7\\
\hline 
3,10&$|\Re\alpha-\Re\beta|$&$7X^3+5X^2+5X+1$&5.952e-4\\
3,20&$|\Re\alpha-\Re\beta|$&$19X^3+8X^2+15X+2$&2.218e-5\\
4,10&$|\Re\alpha-\Re\beta|$&$9X^4+5X^3-X^2+5X-1$&1.472e-6\\
4,20&$|\Re\alpha-\Re\beta|$&$13X^4+3X^3+5X^2+19X-7$&1.669e-7\\
5,10&$|\Re\alpha-\Re\beta|$&$7X^5-6X^4-6X^3-5X^2+X+1$&2.511e-7\\
\hline 
3,10&$|\Im\alpha-\Im\beta|$&$10X^3+6X^2-6X+1$&2.403e-2\\
3,20&$|\Im\alpha-\Im\beta|$&$19X^3+9X^2-19X+5$&5.082e-3\\
4,10&$|\Im\alpha-\Im\beta|$&$10X^4+X^3+10$&6.250e-5\\
4,20&$|\Im\alpha-\Im\beta|$&$20X^4+X^3+20$&7.813e-6\\
5,10&$|\Im\alpha-\Im\beta|$&$5X^5-8X^4+6X^3+5X^2-5X+8$&1.061e-7
\\
\hline
\end{tabular}
\caption{Record polynomials for small degree and height}
\label{tab:records}
\end{table}

The results of an exhaustive search are
displayed in Table~\ref{tab:records}. It appears that the classical
separation seems to be larger than the three other ones (absolute
separation, separation of the real or imaginary parts), whose orders
of growth seem similar.

\begin{table}
\centerline{\begin{tabular}{rlc}
  \hline\\[-1.8ex]
  $\hfil P(X)\hfil$ & abssep & $-\frac{\log
  \operatorname{abssep}}{\log\ho
  (P)}$ \\[1ex]
  \hline\\[-1.8ex]
  $2 X^3+X^2-X-1$ & 5.309e-2 & 4.24 \\
  $13 X^3+11 X^2+8 X+5$ & 3.462e-5 & 4.00 \\
  $8 X^3+7 X^2-9 X-17$ & 2.038e-5 & 3.81 \\
  $17 X^3+9 X^2-7 X-8$ & 1.233e-5 & 3.99 \\
  $17 X^3+9 X^2+7 X+8$ & 1.957e-5 & 3.83 \\
  $102 X^3+97 X^2+71 X+40$ & 1.532e-8 & 3.89 \\
  $153 X^3+97 X^2-71 X-60$ & 4.492e-9 & 3.82 \\
  $71 X^3+112 X^2+153 X+181$ & 1.681e-9 & 3.89 \\
  $181 X^3+153 X^2+112 X+71$ & 9.007e-10 & 4.01 \\[.8ex]
  \hline
  $X^4-X^2-2 X-3$ & 8.615e-4 & 6.42 \\
  $X^4-X^2+2 X-3$ & 8.615e-4 & 6.42 \\
  $3 X^4+3 X^3+X^2-2 X-4$ & 4.585e-5 & 7.21 \\
  $4 X^4+2 X^3-X^2-3 X-3$ & 3.655e-5 & 7.37 \\
  $5 X^4+3 X^3+X^2-X-3$ & 5.893e-5 & 6.05 \\
  $X^4+6 X^3-7 X^2-5 X+6$ & 2.276e-6 & 6.68 \\
  $6 X^4+5 X^3-7 X^2-6 X+1$ & 2.497e-6 & 6.63 \\
  $11 X^4+7 X^3-X^2-10 X-16$ & 2.671e-8 & 6.29 \\
  $16 X^4+10 X^3+X^2-7 X-11$ & 2.266e-8 & 6.35 \\
  $3 X^4+6 X^3-4 X^2+3 X-18$ & 1.799e-8 & 6.17 \\
  $18 X^4+3 X^3+4 X^2+6 X-3$ & 1.095e-8 & 6.34 \\
  $X^4+40 X^3+11 X^2-14 X-55$ & 3.384e-11 & 6.02 \\
  $55 X^4+14 X^3-11 X^2-40 X-1$ & 2.724e-11 & 6.07 \\
  \hline
  $X^5+X^4+2 X^3+3 X^2-2 X+2$ &2.697e-5 & 9.58 \\
  $2 X^5+X^4-X^3+X^2-X-1$ & 1.051e-3 & 9.89 \\
  $2 X^5+X^4+X^3-2 X^2+2 X-2$ & 2.790e-3 & 8.49 \\
  $2 X^5+X^4+X^3-X^2-X-1$ & 3.800e-3 & 8.04 \\
  $2 X^5+X^4+2 X^3-2 X^2+X-2$ & 2.130e-3 & 8.88 \\
  $2 X^5+X^4+2 X^3-X^2-X-2$ & 3.350e-3 & 8.22 \\
  $2 X^5+X^4+2 X^3+2 X^2-2 X-2$ & 2.130e-3 & 8.88 \\
  $2 X^5+2 X^4+X^3-X^2-2$ & 6.473e-4 & 10.59 \\
  $4 X^5+2 X^4-4 X^3+3 X-2$ & 1.463e-6 & 8.03\\
  $8 X^5+5 X^4-4 X^3+4 X^2-5 X-4$ & 5.185e-8 & 8.07 \\
  \hline
\end{tabular}}
\caption{Polynomials with small absolute separation}
\label{table:records2}
\end{table}

For the same degrees and larger height, a fully exhaustive search
becomes too time-consuming. Instead, we performed extensive
experiments. The resulting record values are reported in Table~%
\ref{table:records2}.
For degree 3 (resp. degree
4, degree 5), we computed up to height 200 (resp. height 120, height
30). As expected, polynomials with small separation
tend to have larger height. To balance this bias, according to the
form of the bounds on the separation, we filter the polynomials of
interest by their \emph{quality}, defined as $-\ln(S)/\ln(H)$ for
a polynomial of height $H$ and separation $S$, and display only
polynomials of ``high quality''.

Observing that cubic polynomials of high quality in Table~%
\ref{table:cubic}, namely $13 x^3+11 x^2+8
x+5$ and
$181 x^3+153
x^2+112 x+71$, have roots in similar locations pushed us to refine our
search in this vicinity and eventually led us to an
unexpected family leading to a proof of optimality in the next
section.

For polynomials of degree 4 up to height 120, the results in
Table~\ref{table:records2} seem to suggest a bound of at most $O(\ho
(P)^
{-6})$, which
should be reached by a real-complex gap, as the complex-complex gap
only gives $O(\ho(P)^{-3})$. This is far from the theoretical
bound $O(\ho(P)^{-12})$ from Theorem~\ref{thm:main}. The best
exponent
we obtain by perturbations is~$-5$, see \S\ref{sec:deg5}.

For polynomials of degree 5, we only have results up to height 30
and even those have been obtained by focusing in
some cases of real-complex gap. Here, the bound on the
height seems to be too small to observe a family. We have two
polynomials with good quality that are relatively similar, namely $2
x^5+x^4-x^3+x^2-x-1$ and $8 x^5+5 x^4-4 x^3+4 x^2-5 x-4$. No other polynomial close to these two are observed with height at most 30.

\subsection{The case of degree~3}

\begin{table}
\centerline{\begin{tabular}{rrlc}
  \hline\\[-1.8ex]
  $n$&height& abssep & $-\frac{\log
  \operatorname{abssep}}{\log\ho
  (P)}$ \\[1ex]
  \hline\\[-1.8ex]
2&12&5.093e-3&2.12\\
5&123&2.447e-6&2.68\\
10&2,340&4.643e-11&3.36\\
20&1,694,157& 1.690e-23& 3.66\\
50& 642,934,702,584,732& 8.146e-58& 3.86\\
\hline
\end{tabular}}\caption{Absolute separation for cubic polynomials~$P_n$
from
Proposition~\ref{prop:deg3}}\label{table:cubic}
\end{table}

\begin{proposition}\label{prop:deg3} The family of cubic polynomials
\[        P_n(X)= p_n(3X^3-2X^2+4X-6)+6q_n(X^3-X^2+1)\in\mathbb{Z}[X],\]
where $(p_n/q_n)_n$ is the sequence of convergents of
the
continued
fraction expansion of~$\sqrt{3}$,
has the property that 
\[\abssep(P_n)\ll\ho(P_n)^{-4},\quad n\rightarrow\infty.\]
\end{proposition}
Theorem~\ref{thm:main} shows that $\abssep(P_n)\gg \ho(P_n)^{-4}$, so
 that the exponent~$-4$ is optimal in degree~3. The
 absolute separations of a
 few polynomials in
 that family are displayed in Table~\ref{table:cubic}. The height of
 the polynomials~$P_n$ increases exponentially with~$n$.

\begin{proof}
It is readily checked that the bivariate polynomial
\begin{equation}\label{eq:nicecubicpol}
        P(X,Y)= \left
(-\frac13X^2+\frac12X^3+\frac23X-1\right)Y+X^3-X^2+1
\end{equation}
is such that for $Y=\sqrt{3}$, its 3 roots have absolute value
exactly $\sqrt{3}-1$.
A small perturbation $Y=\sqrt{3}+\epsilon$ sends the real root to 
\[         \sqrt3-1+(2-\sqrt3)\epsilon+O(\epsilon^2),\] 
while the complex roots have absolute value
\[           \sqrt3-1+(2-\sqrt3)\epsilon+O(\epsilon^2), \]
with a different constant in the $O()$ term so that the difference of
the absolute values is $O(\epsilon^2)$.
Evaluating $Y$ at a convergent $p/q$ of the
continued
fraction
expansion of $\sqrt3$ leads to  $\epsilon=|p/q-\sqrt3|<1/q^2$, so
that the difference of the absolute values of roots is $O(q^{-4})$.
The
polynomial with integer coefficients obtained by normalizing $P
(X,p/q)$ has coefficients growing asymptotically like~$q$, giving $-4$
as
a bound in the exponent.
\end{proof}

\paragraph{Further experiments}
In an unsuccessful attempt to generalize the nice family of
cubic polynomials in Proposition~\ref{prop:deg3} to higher degrees, we analyzed
more precisely the
properties of the bivariate polynomial of Equation~\eqref{eq:nicecubicpol}.
A polynomial $a_3X^3+\dots+a_0$ with simple roots, all of the same
nonzero absolute value, has coefficients that satisfy
\begin{equation}\label{eq:cond-eqmod}
a_1^3a_3-a_0a_2^3=0.
\end{equation}
The set of such polynomials is therefore contained in an algebraic set
of dimension~3 in~$\mathbb{R}^4$. The question is to find a good point
on this set whose perturbations behave well with respect to the
absolute separation.

Polynomials in that set factor as $(a_2X+a_1)$ times a quadratic
polynomial, whose discriminant has to be negative for its roots to
have identical absolute value. This gives necessary and sufficient
conditions: $a_1a_2<0\le a_1a_3\le a_2^2\le 3a_1a_3$. 
These conditions define a region on our algebraic set where all
polynomials have 3~roots of identical absolute value.

A very special property possessed by the polynomial of Equation~%
\eqref{eq:nicecubicpol} is that in its case, the polynomial from
Equation~\eqref{eq:cond-eqmod} factors as a square~$(1-Y^2/3)^2$.
Adding this condition and forcing the perturbation of~$Y$ to cancel
the constant and linear coefficients of the expansion of the absolute
separation finally leads to a 3-dimensional set of polynomials of
which that of \eqref{eq:nicecubicpol} is an instance:
\[a(X^3+10c^3)+b(3X^2+6cX)-((X^2+4cX)b+6ac^3)\sqrt{3}.\]
Polynomials in that family all have three roots of identical absolute
value,
and replacing $\sqrt{3}$ by convergents to its continued fraction
expansion lead to an exponent~4 for their asymptotic absolute
separation. 

\subsection{Degrees 4 and 6}

For $d\in\{4,6\}$, we consider the polynomial
$P=M(X^d-1)-Q_d(X)$,
with $M$ an integer and $Q_d$ a
polynomial of degree $d-1$ to be made precise later.
As $M$ tends to infinity, the roots of $P$ tend to those of $X^d-1$.
In particular, one of them tends to~$-1$ and another one tends to $c_d$
a root of the $d$-th cyclotomic polynomial. The asymptotic expansion
in powers of $1/M$ of the difference between the absolute values of
these
roots can be obtained as
follows. First the equation $P=0$ is rewritten as
\begin{equation}\label{eq:perturb}
\frac{X^d-1}{Q_d(X)}=\frac1M.
\end{equation}
In the neighbourhood of a root $\omega$ of $X^d-1$, the 
left-hand side behaves like
\[\frac{d}{\omega Q_d(\omega)}(X-\omega)+\left({\frac{d(d-1)}
{2\omega^2Q_d(\omega)}}-
\frac{dQ_d'(\omega)}{\omega Q_d(\omega)^2}\right)(X-\omega)^2+\dotsb.
\]
Power series inversion then gives the asymptotic behaviour of the
corresponding root of Equation~\eqref{eq:perturb}:
\[X_\omega=\omega+\omega Q_d(\omega)\frac1{dM}+\left
({\omega^2Q_d(\omega)Q_d'(\omega)}-\frac{(d-1)}2\omega Q_d
(\omega)^2\right)\frac1{d^2M^2}+\dotsb\]
Substituting $\omega$ by $1/\omega$ gives the expansion of
the conjugate root and multiplying them gives the expansion of
$|X_\omega|^2$. Finally, subtracting the values of this expansion for
$\omega=-1$ and $\omega=c_d$ gives an expansion of the distance
between the squares of these absolute values
with coefficients that are
polynomials in $c_d$ and in the coefficients of $Q_d$.
Cancelling those coefficients up to order $1/M^d$ gives a system
of $d$ equations in the $d$ coefficients of $Q$.
Up to multiplying $M$ by a constant, when $d\in\{4,6\}$, there is
only one case when this system has
\emph{integer} solutions that do not correspond to $Q_d$ having a
common
factor
with $X^d-1$, leading to the following. 
\begin{proposition}\label{prop:deg4-6}
Let $d\in\{4,6\}$ and let $M$ be a positive integer. Consider the
polynomials
$P_{d,M}$ of degree $d$ defined by 
\[P_{d,M}=M(X^d-1)-Q_d(X),\ \text{with}\ \begin{cases}
Q_4(X)=X^3-X^2+X-5,\\
Q_6(X)=9X^5-9X^4-26X^3-9X^2+9X-28,
\end{cases}\]
As $M$ tends to infinity, these polynomials have height~$M$ and two
roots $\alpha,\beta$ satisfying
\[0<\bigl||\alpha|-|\beta|\bigr|\ll \ho(P_{d,M})^{-d-1}.\]
\end{proposition}
\begin{table}
\centerline{\begin{tabular}{rrlc}
  \hline\\[-1.8ex]
  $d$&height& abssep & $-\frac{\log
  \operatorname{abssep}}{\log\ho
  (P)}$ \\[1ex]
  \hline\\[-1.8ex]
4& 10& 3.716e-5& 4.43\\
4& 20& 4.183e-7& 4.90\\
4& 50& 2.653e-9& 5.05\\
4& 100& 7.175e-11& 5.07\\
4& 500& 2.055e-14& 5.07\\
4& 1000& 6.335e-16& 5.07\\
\hline
6& $10^2$& 3.336e-7& 3.24\\
6& $10^3$& 1.373e-14& 4.62\\
6& $10^4$& 1.267e-21& 5.22\\
6& $10^5$& 1.257e-28& 5.58\\
6& $10^{10}$& 1.256e-63& 6.29\\
6& $10^{20}$& 1.256e-133& 6.65\\
6& $10^{30}$& 1.256e-203& 6.76\\
\hline
\end{tabular}}\caption{Absolute separation for polynomials~$P_{d,M}$
from
Proposition~\ref{prop:deg4-6}}\label{table:deg4-6}
\end{table}
\noindent Absolute separations for a few polynomials in
these families are given in Table~\ref{table:deg4-6}.
\begin{proof}Once these polynomials have been found, the proof can be
carried out by hand (and more easily with the help of a computer
algebra system). The polynomial $P_{4,M}$ has roots
\begin{align*}
z_1&=1-\frac1M-\frac2{M^2}-\frac{11}{2M^3}-\frac{71}{4M^4}+O
\left(\frac1{M^5}\right),\\
z_i&=i-\frac{i}{M}-\frac{1+4i}{2M^2}-\frac{4+11i}{2M^3}-
\frac{66+143i}{8M^4}+O\left(\frac1{M^5}\right).
\end{align*}
Taking the absolute values and subtracting shows that
\[|z_1|-|z_2|\sim\frac{|z_1|^2-|z_i|^2}2=O\left(\frac1{M^5}\right).\]
The same reasoning applies to the polynomial $P_{6,M}$, with roots
\begin{align*}
z_1&=1-\frac{19}{3M}-\frac{1235}{18M^2}-\frac{240445}{162M^3}+\dotsb\\
z_{\frac{-1+i\sqrt{3}}2}&=\frac{-1+i\sqrt3}2-\frac{4(1-i\sqrt3)}
{3M}-\frac{238-166i\sqrt3}{9M^2}-\frac{33383-13511i\sqrt{3}}
{81M^3}+\dotsb
\end{align*}
leading to
\[|z_1|-\left|z_{\frac{-1+i\sqrt{3}}2}\right|\sim
\frac{|z_1|^2-\left|z_{\frac{-1+i\sqrt{3}}2}\right|^2}2=O\left(\frac1
{M^7}\right).\]
\end{proof}
\noindent Perturbations of the other roots, namely $(1\pm i
\sqrt{3})/2$, lead to
another system of equations for the coefficients of~$Q_6$, but this
system does not have integer solutions.

\paragraph{Further experiments}
Looking for similar polynomials using different products of cyclotomic polynomials also yields
\[M(X^2-1)(X^2+X+1)-X^3+3X+4\]
with absolute separation in $\ho^{-5}$.

For $d=8$, the same computation only produces an exponent
$-d+4$: the polynomial system obtained when trying to cancel one more
coefficient does not have any integer solution.

For odd degree $d$, the use of cyclotomic polynomials in this method
only seems to reach exponent $d$ instead of $d+1$, with examples like
\[M(X^2+1)(X-1)+X^2-2X+3,\quad M(X^2-X+1)(X-1)+2X^2-4X-3,\]
in degree~3 and
\begin{align*}
&M(X-1)(X^2+X+1)(X^2-X+1)+3X^4-2X^3+5X^2-4X+7,\\
&M(X-1)(X^2+1)(X^2-X+1)+2X^4-8X^3+13X^2-12X+7
\end{align*}
in degree~5.

\subsection{Degree 5}\label{sec:deg5}
For polynomials of degree 5, exhaustive search cannot reach heights
sufficient to exhibit the beginning of an asymptotic behaviour.

The perturbative approach remains possible, although the polynomial
systems become rapidly too large for the Gr\"obner basis engine we
use, which is Faug\`ere's FGb~\cite{Faugere2010}. We consider the
polynomial $P=MR_5(X)-Q_5(X)$, with
\begin{align*}
R_5(X) &= (X^2 + aX + r^2)(X^2 + bX + r^2), \\
Q_5(X) &= X^5 + q_4 X^4 + q_3 X^3 + q_2 X^2 + q_1 X + q_0,
\end{align*}
where all coefficients are unknowns, and expected to take rational
values, with $a$ and $b$ in the interval $(-2r,2r)$ so that $R_5$ has
four roots of absolute value~$|r|$. As in the previous case, the
equation
$P=0$ rewrites
\[\frac{R_5(X)}{Q_5(X)}=\frac1M.\]
Expanding the left-hand side in the neighbourhood of solutions of
the factors of $R_5$ and inverting the power series expansion gives
expansions of the corresponding perturbed roots $X_1(\epsilon), X_2
(\epsilon)$ of $P$. From there, the coefficients
$c_{1,k}, c_{2,k}$ of $|X_1(\epsilon)|^2$ and $|X_2(\epsilon)|^2$ for
$k$ from $1$ to $5$ are obtained. These coefficients are polynomials
in $r, a, b$ and $q_i$ for $0 \leq i \leq 4$. The polynomial system $
\{ c_{1,k} = c_{2,k}, 1 \leq k \leq 5\}$  is unfortunately too big for
a direct computation by Gr\"obner bases. Instead, we run a loop 
over possible \emph{integer} values of $r, a, b$ in a given search
range, and use Gr\"obner bases to determine whether each
specialized system  has (not necessarily integral) solutions
and to determine the solution. If the solution happens to be rational,
this gives a family with separation $O(\ho(P)^{-6})$. This is how the
following family was found.
\begin{proposition}\label{prop:deg5}Let $M$ be a positive integer and
$P_{5,M}$ be
defined by
\[P_{5,M}=MR_5(X)-Q_5(X),\ \text{with}\ 
\begin{cases}
R_5(X)=(X^2-9X+36)(X^2-11X+36),\\
Q_5(X)=X^5-213X^3+2404X^2-11088X+20736.
\end{cases}
 \]
As $M$ tends to infinity, this polynomial has height $O(M)$ and two
roots $\alpha,\beta$ satisfying
\[0<\bigl||\alpha|-|\beta|\bigr|\ll\ho(P_{5,M})^{-6}.\]
\end{proposition}
\begin{table}
\centerline{\begin{tabular}{rlc}
  \hline\\[-1.8ex]
  height& abssep & $-\frac{\log
  \operatorname{abssep}}{\log\ho
  (P)}$ \\[1ex]
  \hline\\[-1.8ex]
$10^{10}$& 7.165e-38& 3.81\\
$10^{20}$& 7.164e-98& 4.91\\
$10^{50}$& 7.164e-278& 5.56\\
$10^{100}$&7.164e-578& 5.78\\
$10^{200}$&7.164e-1178& 5.89\\
$10^{500}$&7.164e-2978& 5.96\\
\hline
\end{tabular}}\caption{Absolute separation for polynomials~$P_{5,M}$
from
Proposition~\ref{prop:deg5}}\label{table:deg5}
\end{table}

Again, given the polynomials, the proof is a direct computation of the
asymptotic absolute separation. Absolute separations for a few
polynomials in this family are given in Table~\ref{table:deg5}.

Note that actually, we get two one-parameter families of such
polynomials with the same $R_5$ as above and
\begin{align*}
Q_5(X)&=X^5+qX^4-(20q+229)X^3+(2700+171q)X^2\\
&\qquad\qquad\qquad\qquad-(13104+720q)X
+(25920+1296q),\\
Q_5(X)&=X^5+qX^4-(20q+213)X^3+(2404+171q)X^2\\
&\qquad\qquad\qquad\qquad-(11088+720q)X+(20736+1296q).
\end{align*}
The family in the proposition is the special case $q=0$ of the second
one.

\section{Conclusion}
From our experiments in low degree, it would be
tempting to conjecture that a tight bound on the exponent of the
absolute separation in
degree~$d$ is $-d-1$, but we have not been able to prove this, even
for
degrees~7 or~8. Even the degree of the exponent of our lower bounds in
Theorem~\ref{thm:main} seems too high.
The bounds in Theorem~\ref{thm:main} are derived
by considering arbitrary triples or quadruples of roots, without
taking into account that some of them are conjugate, so this
approach could make them pessimistic. More deeply, the cubic
exponent seems to be inevitable for a purely algebraic approach: if
$P=a_0+\dots+a_dz^d$ is written $a_d\prod(z-(x_k+iy_k))$, then
$x_k$, $y_k$ and $x_k^2+y_k^2$ are generically algebraic of degree $d
(d-1)/2$ over
$\mathbb{Q}(a_0,\dots,a_d)$, with minimal polynomials having degree
$d-1$ in the $a_i$'s, whence again a cubic exponent.
Obtaining better bounds thus probably needs more analytic tools.

Another question which is of a more analytic nature and more directly
relevant for the complexity of algorithms on linear recurrences would
be to determine a bound on the minimal distance between the
two \emph{largest} distinct absolute
values of the roots. For any polynomial, this quantity is
at least as large as the absolute separation, but does it
have a different asymptotic behaviour?

On a related matter, Koiran recently used analytic arguments (Rolle's
theorem and Baker's theory of linear forms in the logarithms of algebraic numbers) 
to give a bound on the (classical) root separation for trinomials,
with a very small dependency on the degree~\cite{Koiran2019}. It is
not clear to us whether similar results also hold for the absolute
separation.

\section*{Ackowledgements} W.~Fang and B.~Salvy were supported in
part by FastRelax
ANR-14-CE25-0018-01. A.~Dujella and T.~Pejkovi\'c were supported
by the Croatian Science
Foundation under the project no. IP-2018-01-1313 and the QuantiXLie
Center of Excellence, a project co-financed by the Croatian Government
and European Union through the European Regional Development Fund -
the Competitiveness and Cohesion Operational Programme (Grant
KK.01.1.1.01.0004).
Y.~Bugeaud, A. Dujella, T. Pejkovi\'{c}, and B. Salvy were supported
in part by the French-Croatian bilateral 
COGITO project `Approximation diophantienne et courbes elliptiques'.


\end{document}